\documentclass[12pt]{amsart}
\usepackage[latin1]{inputenc}
\usepackage[english]{babel}
\usepackage{amsmath,amssymb,amsthm,amsfonts, color}
\usepackage{enumerate}
\usepackage{graphicx}
\usepackage{latexsym}
\usepackage[all]{xy}
\usepackage{tikz}

\makeatletter
\@namedef{subjclassname@2010}{ \textup{2010} Mathematics Subject Classification}
\makeatother

\newtheorem{theorem}{Theorem}
\newtheorem{corollary}{Corollary}
\newtheorem{propos}{Proposition}
\newtheorem{remark}{Remark}

\theoremstyle{definition}

\DeclareMathOperator{\supp}{supp}

\begin{document}
\title{A version of Calder\'{o}n-Mityagin theorem for the class of rearrangement invariant groups.}
\thanks{{\rm *}The author has been supported by the Ministry of Education and Science 
of the Russian Federation (project 1.470.2016/1.4) and by the RFBR grant 18--01--00414.}
\author[Astashkin]{Sergey V. Astashkin}
\address[Sergey V. Astashkin]{Department of Mathematics, Samara National Research University, Moskovskoye shosse 34, 443086, Samara, Russia
}
\email{\texttt{astash56@mail.ru}}
\maketitle

\vspace{-7mm}

\begin{abstract}
Let $l_0$ be the group (with respect to the coordinate-wise addition) of all sequences of real numbers $x=(x_k)_{k=1}^\infty$ that are eventually zero, equipped with the quasi-norm $\|x\|_0={\rm{card}}\{\supp\,x\}$. A description of orbits of elements in the pair $(l_0,l_1)$ is given, which   complements (in the sequence space setting) the classical Calder\'{o}n-Mityagin theorem on a description of orbits of elements in the pair $(l_1,l_\infty)$. As a consequence, we obtain that the pair $(l_0,l_1)$ is ${\mathcal K}$-monotone. 
\end{abstract}


\section{\protect \medskip Introduction, preliminaries and main results}

\newcommand{\Orb}{{\rm{Orb}}\,}

According to the classical Calder\'{o}n-Mityagin theorem (see \cite{Cal}, \cite{Mit}), if $b=(b_k)_{k=1}^\infty\in l_\infty$, then a sequence $a=(a_k)_{k=1}^\infty$ is representable in the form $a=Tb$ for some linear operator $T$ bounded both in $l_\infty$ and $l_1$ if and only if
$$
\sum_{i=1}^k a_i^*\le C\sum_{i=1}^k b_i^*,\;\;k=1,2,\dots$$
for some $C>0$, where $(u_k^*)_{k=1}^\infty$ is the nonincreasing rearrangement of the sequence $(|u_k|)_{k=1}^\infty$. Here, we give a constructive proof of a counterpart of this result for the class of rearrangement invariant groups, intermediate between $l_1$ and the group $l_0$ of eventually zero sequences with a natural quasi-norm. 

Let $X$ be an Abelian group of sequences of real numbers $x=(x_k)_{k=1}^\infty$ with respect to the coordinate-wise
addition. Recall that a quasi-norm on $X$ is a real function $x\mapsto \|x\|_X$ that  satisfies the conditions: (a)
$\|x\|_X\ge 0$ and $\| x\|_X=0\Longleftrightarrow x=0$;
(b) $\|-x\|_X=\| x\|_X$; (c) $\| x+y\|_X\le C(\| x\|_X+\| y\|_X)$ for some $C\ge 1$. In this case X is called a {\it quasi-normed group}. As is known (see e.g. \cite[Lemma~3.10.1]{BL}), without loss of generality, we may assume that $C=1$. It will be supposed also that the quasi-norm $x\mapsto \|x\|_X$ has the ideal and rearrangement invariant properties that can be expressed as follows: if sequences $x=(x_k)_{k=1}^\infty$ and $y=(y_k)_{k=1}^\infty$ are such that $x_k^*\le y_k^*$, $k=1,2,\dots$ and $y\in X$, then $x\in X$ and
$\|x\|_X\le\|y\|_X$. Further, we shall refer such a quasi-normed group as to a {\it rearrangement invariant (r.i.) group}. An important example  is the classical space $l_p$, $0<p\le\infty$, where for $0 < p < 1$ the quasi-norm on $l_p$ is defined by 
$$
\|x\|_p:=\sum_{k=1}^\infty |x_k|^p.
$$
Passing to the limit as $p\to 0$ in this formula, we get the set $l_0$ of all sequences $x=(x_k)_{k=1}^\infty$ that are eventually zero, i.e., such that $\|x\|_0:={\rm{card}}\{\supp\,x\}<\infty$, where $\supp\,x:=\{k\in\mathbb{N}:\, x_k\ne 0\}$.
Then $l_0$ becomes a r.i. group equipped with the sub-additive quasi-norm $\|x\|_0$, which generates on $l_0$ the discrete topology. Since $l_0$ is continuously embedded into $l_1$, the pair $(l_0,l_1)$ is compatible. Recall that a  pair of quasi-normed groups $(X_0,X_1)$ is {\it compatible} if $X_0$ and $X_1$ are embedded continuously into some Hausdorff topological group.

The main goal of this paper is to get a description of the orbit $\Orb(b;l_0,l_1)$ in the pair $(l_0,l_1)$ of an arbitrary element $b\in l_1$ and by using this result to obtain some interpolation properties of this pair. Observe that similar problems for the pair $(l_0,l_\infty)$ (more precisely, for the corresponding pair of r.i. function groups) were considered in \cite{A11}.

Recall first some necessary definitions (for more detailed information we refer to the paper \cite{PS} and the monographs \cite{BL,BK,BS,KPS}). 

A mapping $T:\,X\to X$, where $X$ is a quasi-normed group, is called a {\it homomorphism} on $X$ if $T(x+y)=Tx+Ty$ and $T(-x)=-T(x)$ for $x,y\in X$.
As usual, a homomorphism $T$ is called {\it bounded} if
$$ 
\|T\|_{X\to X}:= \sup\limits_{x \ne 0}
{\|Tx \|\over \|x \|}< \infty.
$$
Let $(X_0,X_1)$ be a compatible pair of quasi-normed groups. We define the {\it orbit} of an element $b\in X_0+X_1$ with respect to $(X_0,X_1)$ as the set $\Orb(b;X_0,X_1)$ of all $x\in X_0+X_1$, representable in the form $x=Tb$, where $T$ is a bounded homomorphism in $X_0$ and $X_1$. Furthermore, we let
$$
\|x\|_{\Orb}:=\|x\|_{\Orb(b;X_0,X_1)}=\inf\| T\|_{(X_0,X_1)},
$$
where $\|T\|_{(X_0,X_1)}:= \max\limits_{i=0,1}\|T\|_{X_i\to X_i}$ and the infimum is taken over all homomorphisms $T$ such that $Ta=x$. It is easy to check that the quasi-norm $x\mapsto \|x\|_{\Orb}$
makes $\Orb(b;X_0,X_1)$ into a r.i. group.

We can now state our main result.

\begin{theorem}\label{Th1}
Let $b=(b_i)_{i=1}^\infty\in l_1$. Then, a sequence $a=(a_i)_{i=1}^\infty\in l_1$ belongs to the orbit $\Orb(b;l_0,l_1)$ if and only if there is a constant $C>0$ such that for all $k=1,2,\dots$ we have 
\begin{equation}\label{EQ12}
\sum_{i=k}^\infty a_i^*\le C\sum_{i=[k/C]}^\infty b_i^*,\;\;k=1,2,\dots
\end{equation}

Moreover, \eqref{EQ12} holds with $C=\|a\|_{\Orb}$, whenever $a\in\Orb(b;l_0,l_1)$. Conversely, if we have \eqref{EQ12}, then $\|a\|_{\Orb}\le 9(1+[C])$, for $C>1$, and $\|a\|_{\Orb}\le 3[C^{-1}]^{-1}$, for $C\le 1$.
\end{theorem}


Given a compatible pair of quasi-normed groups $(X_0, X_1)$, we introduce the approximative ${\mathcal E}$-functional by
$$
{\mathcal E}(t,x;X_0,X_1):=\inf\{\|x-x_0\|_{X_1}:\,x_0\in X_0,x-x_0\in X_1,\|x_0\|_{X_0}\le t\},\quad 
x\in X_0+X_1,\ t>0
$$
\cite[Chapter~7]{BL}. Clearly, $t\mapsto {\mathcal E}(t,x;X_0,X_1)$ is a decreasing function on $[0,\infty)$. Furthermore, for every $b\in X_0+X_1$ we define the {\it $\mathcal E$-orbit} of $b$ in the pair  $(X_0,X_1)$ as the set ${\mathcal E}-\Orb(b;X_0,X_1)$ of all $x\in X_0+X_1$ such that for some $C>0$
$$
{\mathcal E}(t,x;X_0,X_1)\le C {\mathcal E}(t/C,b;X_0,X_1),\;\;t>0.$$
Then, the functional 
$\|x\|_{{\mathcal E}-\Orb}:=\|x\|_{{\mathcal E}-\Orb(b;X_0,X_1)}$, which is equal to the infimum of $C$ satisfying the last inequality, is a quasi-norm on the group ${\mathcal E}-\Orb(b;X_0,X_1)$. To show this, it suffices to check that this functional is sub-additive provided that the quasi-norms in $X_0$ and $X_1$ have the latter property.

Indeed, suppose that $C_1>\|x\|_{{\mathcal E}-\Orb}$ and $C_2>\|y\|_{{\mathcal E}-\Orb}$. Then, by \cite[Lemma~7.1.1]{BL}, for each $0<\gamma<1$ we have 
\begin{eqnarray*}
{\mathcal E}(t,x+y;X_0,X_1)&\le& {\mathcal E}(\gamma t,x;X_0,X_1)+{\mathcal E}((1-\gamma) t,y;X_0,X_1)\\
&\le& C_1{\mathcal E}(\gamma t/C_1,b;X_0,X_1)+C_2{\mathcal E}((1-\gamma) t/C_2,b;X_0,X_1).
\end{eqnarray*}
Hence, choosing $\gamma=C_1/(C_1+C_2)$, we infer
$$
{\mathcal E}(t,x+y;X_0,X_1)\le (C_1+C_2){\mathcal E}(t/(C_1+C_2),b;X_0,X_1),\;\;t>0,
$$
and, since $C_1>\|x\|_{{\mathcal E}-\Orb}$ and $C_2>\|y\|_{{\mathcal E}-\Orb}$ are arbitrary, we conclude that
$$
\|x+y\|_{{\mathcal E}-\Orb}\le  \|x\|_{{\mathcal E}-\Orb}+\|y\|_{{\mathcal E}-\Orb}.$$ 
The rest of the properties of a r.i. quasi-norm for the functional $x\mapsto \|x\|_{{\mathcal E}-\Orb}$ follows immediately from the definition. 

It is clear that for every $x=(x_i)_{i=1}^\infty\in l_1$ we have
\begin{equation}\label{EQ13d}
{\mathcal E}(t,x;l_0,l_1)=\inf\{\|x-x_0\|_{1}:\,{\rm{card}}\{\supp\,x_0\}\le t\}=\sum_{i=[t]+1}^\infty x_i^*.
\end{equation}
As a consequence of Theorem \ref{Th1}, we obtain

\begin{corollary}\label{cor0}
For every $b=(b_i)_{i=1}^\infty\in l_1$ 
$$
\Orb(b;l_0,l_1)=  \mathcal E-\Orb(b;l_0,l_1).$$
Moreover, if $\|x\|_{{\mathcal E}-\Orb}:=\|x\|_{{\mathcal E}-\Orb}(b;l_0,l_1)$, $\|x\|_{\Orb}:=\|x\|_{\Orb}(b;l_0,l_1)$, we have $\|x\|_{{\mathcal E}-\Orb}\le \|x\|_{\Orb}$, and $\|x\|_{\Orb}\le 9([\|x\|_{{\mathcal E}-\Orb}]+1)$ if $\|x\|_{{\mathcal E}-\Orb}>1$, $\|x\|_{\Orb}\le 3[\|x\|_{{\mathcal E}-\Orb}^{-1}]^{-1}$ if $\|x\|_{{\mathcal E}-\Orb}\le 1$.
\end{corollary}

One of the main problems of the interpolation theory of operators is a description of the class of interpolation spaces (groups) with respect to the given compatible pair of spaces (groups). A quasi-normed group $X$ is said to be {\it interpolation} with respect to a pair of quasi-normed groups $(X_0,X_1)$ if $X_0\cap X_1\subset X\subset X_0+X_1$ and every homomorphism bounded in $X_0$ and $X_1$ is bounded in $X$ as well. 
It is well known \cite{BK,BS} that this problem can be effectively resolved for the so-called $\mathcal K$-monotone pairs. 

Let $(X_0,X_1)$ be a compatible pair of quasi-normed groups, $b\in X_0+X_1$. Then, the {\it $\mathcal K$-orbit} of $b$ in the pair  $(X_0,X_1)$ is the set
${\mathcal K}-\Orb(b;X_0,X_1)$ of all $x\in X_0+X_1$ such that
$$
\|x\|_{{\mathcal K}-\Orb}:=\sup\limits_{t>0}
{{\mathcal K}(t,x;X_0,X_1)\over {\mathcal K}(t,b;X_0,X_1)}<\infty ,
$$
where ${\mathcal K}(t,x;\,X_0,X_1)$ is the so-called Peetre's $\mathcal K$-functional defined by 
$$
{\mathcal K}(t,x;\,X_0,X_1):=\inf\{\|x_0\|_{X_0}+t\|x_1\|_{X_1};\,
x=x_0+x_1,x_i\in X_i\}.
$$
A pair $(X_0,X_1)$ is called {\it ${\mathcal K}$-monotone} if
${\mathcal K}-\Orb(b;X_0,X_1)=\Orb(b;X_0,X_1)$ for every
$b\in X_0+X_1$. Historically, the first example of such a pair was the pair $(L_1,L_\infty)$ of functions on a $\sigma$-finite measure space (in particular, the pair of sequence spaces $(l_1,l_\infty)$;  see \cite{Cal}, \cite{Mit}). For further examples of ${\mathcal K}$-monotone pairs and also their properties we refer to the monographs \cite{BL,BK,BS}.


\begin{corollary}\label{cor1}
For every $b=(b_i)_{i=1}^\infty\in l_1$ we have
$$
\Orb(b;l_0,l_1)=  \mathcal K-\Orb(b;l_0,l_1).$$
Also, if $\|x\|_{{\mathcal K}-\Orb}:=\|x\|_{{\mathcal K}-\Orb}(b;l_0,l_1)$, $\|x\|_{\Orb}:=\|x\|_{\Orb}(b;l_0,l_1)$, we have $\|x\|_{{\mathcal K}-\Orb}\le\|x\|_{\Orb}$, and $\|x\|_{\Orb}\le 9([2\|x\|_{{\mathcal K}-\Orb}]+1)$ if $\|x\|_{{\mathcal K}-\Orb}>1/2$, $\|x\|_{\Orb}\le 3[(2\|x\|_{{\mathcal K}-\Orb})^{-1}]^{-1}$ if $\|x\|_{{\mathcal K}-\Orb}\le 1/2$.

Therefore, the pair $(l_0,l_1)$ is ${\mathcal K}$-monotone. 
\end{corollary}

\begin{corollary}\label{cor2}
Let $X$ be a r.i. sequence group such that $l_0\subset X\subset l_1$.  Then, $X$ is an interpolation group with respect to the pair $(l_0,l_1)$ if and only if from the inequality 
$$
\mathcal K(t,x;l_0,l_1)\le C'\mathcal K(t,b;l_0,l_1),\;\;t>0,$$
for some $C'$, and $b\in X$ it follows that $x\in X$ and $\|x\|_X\le C\|b\|_X$, where $C>0$ does not depend on $x$ and $b$.
\end{corollary}

In what follows, $[\alpha]$ is the integer part of a real number $\alpha$. Moreover, if $I$ and $J$ are subsets of $\mathbb{N}$ such that $i<j$ for all $i\in I$, $j\in J$, then we write $I<J$. In particular, instead of $I<\{j\}$ (resp. $\{i\}<J$) we shall write $I<j$ (resp. $i<J$).

\section{\protect \medskip Proofs}

We begin with proving some auxiliary results.

\begin{propos}\label{prop1}
Let $b=(b_i)_{i=1}^\infty\in l_1$. If $a=(a_i)_{i=1}^\infty\in \Orb(b;l_0,l_1)$, then  
\begin{equation}\label{EQ14d}
\sum_{i=k}^\infty a_i^*\le \|a\|_{\Orb}\cdot\sum_{i=[k/\|a\|_{\Orb}]}^\infty b_i^*,\;\;k=1,2,\dots
\end{equation}
\end{propos}
\begin{proof}
First, let $(X_0, X_1)$ be a pair of quasi-normed groups, and let $T$ be a bounded homomorphism in $X_0$ and $X_1$. Then,
\begin{eqnarray*}
{\mathcal E}(t,Tx;X_0,X_1)&=&\inf\{\|Tx-y_0\|_{X_1}:\,\|y_0\|_{X_0}\le t\}\\
&\le& \inf\{\|Tx-Tx_0\|_{X_1}:\,\|x_0\|_{X_0}\le t/\|T\|\}\\ &\le& \|T\| {\mathcal E}(t/\|T\|,x;X_0,X_1).
\end{eqnarray*}

Since $a=(a_i)_{i=1}^\infty\in \Orb(b;l_0,l_1)$, for every $\varepsilon>0$ there exists a homomorphism $T$ such that $Tb = a$ and $\|T\|:=\|T\|_{(l_0,l_1)}\le\|a\|_{\Orb}+\varepsilon $. Then, applying the preceding estimate for the pair $(l_0,l_1)$ and taking into account equation \eqref{EQ13d}, we obtain
$$
\sum_{i=[t]+1}^\infty a_i^*\le\|T\|\sum_{i=[t/\|T\|]+1}^\infty b_i^* \le(\|a\|_{\Orb}+\varepsilon )\sum_{i=[t/(\|a\|_{\Orb}+\varepsilon)]+1}^\infty b_i^*,\;\;t>0.
$$
Since $\varepsilon> 0$ is arbitrary, this inequality implies \eqref{EQ14d}.
\end{proof}

The following statement will play a key role in the future.

\begin{propos}\label{prop2}
Let $a=(a_i)_{i=1}^\infty$ and $b=(b_i)_{i=1}^\infty\in l_1$ be two sequences satisfying the inequality 
\begin{equation}\label{EQ13aa}
\sum_{i=k}^\infty a_i^*\le \sum_{i=k}^\infty b_i^*,\,\,k=1,2,\dots
\end{equation}
Then, there exists a homomorphism $Q:\,l_1\to l_1$ such that $\|Q\|_{l_1\to l_1}\le 2$,  $\|Q\|_{l_0\to l_0}\le 3$, and $Qb=a$. 
\end{propos}
\begin{proof}
Clearly, there are linear operators $T_k$, $k=1,2$, such that $T_1((a_i^*))=a$ and $T_2b=(b_i^*)$ with $\|T_k\|_{l_1\to l_1}=\|T_k\|_{l_0\to l_0}=1$, $k=1,2$. Therefore, without loss of generality we may assume that $a_i=a_i^*$ and $b_i=b_i^*$, $i=1,2,\dots$ Hence, \eqref{EQ13aa}  may be rewritten as follows:
\begin{equation}\label{EQ13a}
\sum_{i=k}^\infty a_i\le \sum_{i=k}^\infty b_i,\,\,k=1,2,\dots
\end{equation}

Let $J:=\{i:\,a_i>2b_i\}$, $I:=\{i:\,a_i<b_i\}$, and $K:=\{i:\,b_i\le a_i\le 2b_i\}$. Clearly, $\mathbb{N}=J\cup I\cup K$. If $J=\emptyset$, then to conclude the proof it is suffices to take the operator 
$$
Qx=\sum_{i\in I}u_ix_ie_i+ \sum_{i\in K}(1+v_i)x_ie_i,\;\;x=(x_i)_{i=1}^\infty\in l_1,
$$
where $\{e_i\}_{i=1}^\infty$ is the unit vector basis in $l_1$, $u_i:=a_i/b_i$, $i\in I$, and $v_i:=(a_i-b_i)/b_i$, $i\in K$. Indeed, one can easily check that 
$\|Q\|_{l_1\to l_1}\le 2$, $\|Q\|_{l_0\to l_0}=1$ and $Qb=a$. 

Now, suppose that $J=\{j_k\}_{k=1}^{k_0}$, where $1\le k_0\le\infty$ and $j_1<j_2<\dots$ Denoting $\delta_k:=a_{j_k}-b_{j_k}$, $1\le k\le k_0$ and $\eta_i:=b_{i}-a_{i}$, $i\in I$, we claim that for each $1\le m\le k_0$
\begin{equation}\label{EQ13}
\sum_{i\in I, i>j_m} \eta_i\ge \sum_{m\le k\le k_0} \delta_k+ \Big(\delta_{m-1}-\sum_{i\in I, j_{m-1}<i<j_m} \eta_i\Big)_+,
\end{equation}
where we set $\delta_0=0$ and $j_0=1$. 

Since for $m=1$ inequality \eqref{EQ13} is an easy consequence of inequality \eqref{EQ13a} for $k=j_1$, we can assume that
$2\le m\le k_0$. First, from the notation and \eqref{EQ13a} it follows that 
\begin{equation}\label{EQ14}
\sum_{i\in I, i>j_m} \eta_i\ge
\sum_{i\ge j_m}(b_i-a_i)+\sum_{m\le k\le k_0} \delta_k\ge \sum_{m\le k\le k_0} \delta_k.
\end{equation}
Therefore, we obtain \eqref{EQ13} for $m$ such that 
$$
\delta_{m-1}\le\sum_{i\in I, j_{m-1}<i<j_m} \eta_i.
$$
Otherwise, assume that for some $2\le m\le k_0$ 
$$
\delta_{m-1}>\sum_{i\in I, j_{m-1}<i<j_m} \eta_i$$
and \eqref{EQ13} does not hold. Then, we have
$$
\sum_{i\in I,i>j_{m-1}} \eta_i< \sum_{k=m-1}^{k_0} \delta_k,
$$
whence
$$
\sum_{i=j_{m-1}}^\infty(b_i-a_i)\le \sum_{i\in I,i>j_{m-1}} \eta_i- \sum_{k=m-1}^{k_0} \delta_k<0.
$$
Since this contradicts  \eqref{EQ13a} for $k=j_{m-1}$,  inequality \eqref{EQ13} is proved. 

Let $l_1:=\min\{i\in I:\,i>j_1\}$. Thanks to \eqref{EQ13}, we can choose $i_1\in I$, $i_1>j_1$ so that
$$
\sum_{i\in I,j_1<i< i_1} \eta_i<\delta_1\le\sum_{i\in I,j_1<i\le i_1} \eta_i.$$ 
Observe that $i_1>l_1$. Indeed, by the definition of the set $J$, for each $i\in I$, $i>j_1$ we have 
$$
\delta_1=a_{j_1}-b_{j_1}>b_{j_1}\ge b_i\ge\eta_{i}.$$
Therefore, the set $I_1:=\{i\in I:\,l_1\le i\le m_1\}$, where $m_1:=\max\{i\in I:\,i<i_1\}$, is not empty. Setting
$$
\eta'_{i_1}:=\delta_1-\sum_{i\in I_1} \eta_i,$$
we have 
\begin{equation}\label{EQ16}
0<\eta'_{i_1}\le \eta_{i_1}=b_{i_1}-a_{i_1}\;\;\mbox{and}\;\;\sum_{i\in I_1} \eta_i+\eta'_{i_1}=\delta_1.
\end{equation}
Next, we consider the cases when $i_1<j_2$ and $i_1>j_2$ separately. 

If $i_1<j_2$, then $\delta_1\le \sum_{i\in I,j_1<i<j_2} \eta_i$. Since from \eqref{EQ13} it follows that 
$$
\sum_{i\in I,i>j_2}\eta_i\ge \sum_{2\le k\le k_0} \delta_k,
$$
we can proceed next as above, setting $l_2:=\min\{i\in I:\,i>j_2\}$ and choosing $i_2\in I$, $i_2>j_2$ so that
$$
\sum_{i\in I,j_2<i< i_2} \eta_i<\delta_2\le\sum_{i\in I,j_2<i\le i_2} \eta_i.$$ 
Again  $i_2>l_2$ and hence the set $I_2:=\{i\in I:\,l_2\le i\le m_2\}$, where $m_2:=\max\{i\in I:\,i<i_2\}$, is not empty. Then, if
$$
\eta'_{i_2}:=\delta_2-\sum_{i\in I_2} \eta_i,$$
we have 
\begin{equation}\label{EQ16'}
0<\eta'_{i_2}\le \eta_{i_2}=b_{i_2}-a_{i_2}\;\;\mbox{and}\;\;\sum_{i\in I_2} \eta_i+\eta'_{i_2}=\delta_2.
\end{equation}

Let now $i_1>j_2$. In this case 
$$
\delta_1-\sum_{i\in I,j_1<i<j_2} \eta_i=\sum_{i\in I,j_2<i<i_1} \eta_i+\eta'_{i_1}.$$
Then, by \eqref{EQ13}, we have 
$$
\sum_{i\in I,i>j_2}\eta_i\ge \sum_{2\le k\le k_0} \delta_k+\sum_{i\in I,j_2<i<i_1} \eta_i+\eta'_{i_1},
$$
whence
$$
\sum_{i\in I,i>i_1}\eta_i+\eta''_{i_1}\ge \sum_{2\le k\le k_0} \delta_k,
$$
where $\eta''_{i_1}:=\eta_{i_1}-\eta'_{i_1}\ge 0$. Also, from the definition of the set $J$ and inequality $j_2<i_1$ it follows
$$
\delta_2=a_{j_2}-b_{j_2}>b_{j_2}\ge b_{i_1}\ge \eta''_{i_1}.$$
Therefore, setting $l_2:=\min\{i\in I:\,i>i_1\}$, we can find $i_2\in I$, $i_2\ge l_2>i_1$ such that
$$
\eta''_{i_1}+\sum_{i\in I,l_2\le i< i_2} \eta_i<\delta_2\le \eta''_{i_1}+\sum_{i\in I,l_2\le i\le i_2} \eta_i.$$ 
In the case when $i_2=l_2$ we put $I_2=\emptyset$. If  $i_2>l_2$, we define $I_2:=\{i\in I:\,l_2\le i\le m_2\}$, where $m_2:=\max\{i\in I:\,i<i_2\}$. Then, if
$$
\eta'_{i_2}:=\delta_2-\sum_{i\in I_2} \eta_i-\eta''_{i_1},$$
we have
\begin{equation}\label{EQ16''}
\sum_{i\in I_2} \eta_i+\eta'_{i_2}+\eta''_{i_1}=\delta_2.
\end{equation}
Also, from the definition of $\eta''_{i_1}$ it follows 
\begin{equation}\label{EQ19}
\eta'_{i_1}+\eta''_{i_1}=\eta_{i_1}=b_{i_1}-a_{i_1}.
\end{equation}

Next, we proceed similarly considering again two different cases when $i_2<j_3$ and $i_2>j_3$ separately and using inequality  \eqref{EQ13}.

Suppose first that $k_0=\infty$. 
As a result of the above procedure, we get sets $I_k:=\{i\in I:\,l_k\le i\le m_k\}$, $k=1,2,\dots$ (some of them may be empty), sequences $\{i_j\}_{j=1}^\infty\subset I$,
 $\{\eta'_{i_j}\}_{j=1}^\infty$ and $\{\eta''_{i_j}\}_{j=1}^\infty$ such that $I_1<i_1<I_2<i_2<\dots$ and 
\begin{equation}\label{EQ21}
\sum_{i\in I_k} \eta_i+\eta'_{i_k}+\eta''_{i_{k-1}}=\delta_k,\;\;k=1,2,\dots,
\end{equation} 
\begin{equation}\label{EQ20}
\eta'_{i_k}+\eta''_{i_k}\le b_{i_k}-a_{i_k},\;\;k=1,2,\dots,
\end{equation}
where $\eta''_{i_k}=0$ if $i_{k}<j_{k+1}$, $k=1,2,\dots$ (see \eqref{EQ16} --- \eqref{EQ19}).

Denote $h_i:=\eta_i/b_i$, $i\in \cup_{k=1}^\infty I_k$, $h'_{i_j}:=\eta'_{i_j}/b_{i_j}$, $j=1,2,\dots$, and 
\begin{equation}\label{EQ21.5}
h''_{i_j}:=\frac{\eta''_{i_j}}{b_{i_j}-\eta'_{i_j}}=\frac{\eta''_{i_j}}{b_{i_j}(1-h'_{i_j})},\;\;j=1,2,\dots
\end{equation}
Then from \eqref{EQ21} it follows that 
\begin{equation}\label{EQ22}
\sum_{i\in I_k} h_ib_i+h'_{i_k}b_{i_k}+h''_{i_{k-1}}(1-h'_{i_{k-1}}) b_{i_{k-1}}=\delta_k,\;\;k=1,2,\dots,
\end{equation}
where $h'_{i_0}=h''_{i_0}=0$. Moreover, let  $u_i:=a_i/b_i$, $i\in I':=I\setminus \cup_{k=1}^\infty (I_k\cup \{i_k\})$, and $v_i:=(a_i-b_i)/b_i$, $i\in K$. Now, define the linear operator $T: l_1\to l_1$ by
\begin{eqnarray}
Tx:&=&\sum_{k=1}^\infty\Big(x_{j_k}+\sum_{i\in I_k} h_ix_{i}+h'_{i_k}x_{i_k}+h''_{i_{k-1}}(1-h'_{i_{k-1}}) x_{i_{k-1}}\Big)e_{j_k}\nonumber\\
&+&\sum_{k=1}^\infty\sum_{i\in I_k} (1-h_i)x_{i}e_i
+\sum_{k=1}^\infty(1-h'_{i_k})(1-h''_{i_k})x_{i_k}e_{i_k}\nonumber\\ &+& \sum_{i\in I'}u_ix_ie_i+ \sum_{i\in K}(1+v_i)x_ie_i,\;\;x=(x_i)_{i=1}^\infty\in l_1,
\label{EQ23}
\end{eqnarray}
where $\{e_i\}_{i=1}^\infty$ is the unit vector basis in $l_1$.

Since $0\le h_i<1$, $i\in I_k$, $0\le h'_{i_k}, h''_{i_k}< 1$, $k=1,2,\dots$, $0\le u_i< 1$, $i\in I'$, $0\le v_i\le 1$, $i\in K$, then for every $x=(x_i)_{i=1}^\infty\in l_1$ we have 
$$
\|Tx\|_{l_1}\le \sum_{i\not\in K}|x_{i}|+
\sum_{i\in K} (1+v_i)|x_{i}|\le 2\|x\|_{l_1},$$ 
whence $\|T\|_{l_1\to l_1}\le 2$. 

Further, if $i\ne j_k$, $k=1,2,\dots$, then $x_i=0$ if and only if $(Tx)_i=0$. Moreover, from $(Tx)_{j_k}\ne 0$ it follows that at least one of the following relations holds: $x_{j_k}\ne 0$, or $x_i\ne 0$ for some $i\in I_k$, or $x_{i_k}\ne 0$, or $x_{i_{k-1}}\ne 0$. Since $I_{k_1}\cap I_{k_2}=\emptyset$ if $k_1\ne k_2$, we conclude that
$$
{\rm card}\{i:\,(Tx)_i\ne 0\}\le 3 {\rm card}\{i:\,x_i\ne 0\},$$
and so $\|T\|_{l_0\to l_0}\le 3$. At last, from \eqref{EQ20}, \eqref{EQ21.5} and \eqref{EQ22} it follows that $Tb=a$.

Now, suppose that $k_0<\infty$. Then, as above, we construct sets $I_k:=\{i\in I:\,l_k\le i\le m_k\}$, $k<k_0$ (some of them may be empty), sequences $\{i_j\}_{j<k_0}\subset I$,
 $\{\eta'_{i_j}\}_{j<k_0}$ and $\{\eta''_{i_j}\}_{j<k_0}$ such that $I_1<i_1<I_2<i_{2}<\dots<I_{k_0-1}<i_{k_0-1}$ and 
\begin{equation}\label{EQ21*}
\sum_{i\in I_k} \eta_i+\eta'_{i_k}+\eta''_{i_{k-1}}=\delta_k,\;\;k<k_0,
\end{equation} 
\begin{equation}\label{EQ20*}
\eta'_{i_k}+\eta''_{i_k}\le b_{i_k}-a_{i_k},\;\;k<k_0,
\end{equation}
where $\eta''_{i_k}=0$ provided that $i_{k}<j_{k+1}$, $k<k_0-1$. 

If $i_{k_0-1}<j_{k_0}$, we have
$$
\delta_{k_0-1}\le \sum_{i\in I,j_{k_0-1}<i<j_{k_0}} \eta_i,$$
and hence from \eqref{EQ13} it follows that 
\begin{equation}\label{EQ24}
\sum_{i\in I,i>j_{k_0}}\eta_i\ge \delta_{k_0}.
\end{equation}

Otherwise, we have $i_{k_0-1}>j_{k_0}$.  Then, 
$$
\delta_{k_0-1}-\sum_{i\in I,j_{k_0-1}<i<j_{k_0}} \eta_i=\sum_{i\in I,j_{k_0}<i<i_{k_0-1}} \eta_i+\eta'_{i_{k_0-1}}.$$
Therefore, by \eqref{EQ13}, in this case we have 
$$
\sum_{i\in I,i>j_{k_0}}\eta_i\ge \delta_{k_0}+\sum_{i\in I,j_{k_0}<i<i_{k_0-1}} \eta_i+\eta'_{i_{k_0-1}},
$$
or
\begin{equation}\label{EQ25}
\sum_{i\in I,i>i_{k_0-1}}\eta_i+\eta''_{i_{k_0-1}}\ge \delta_{k_0},
\end{equation}
because $\eta''_{i_{k_0-1}}=\eta_{i_{k_0-1}}-\eta'_{i_{k_0-1}}$.

We set now $I_{k_0}:=\{i\in I:\, i>j_{k_0}\}$ (resp. $I_{k_0}:=\{i\in I:\, i>i_{k_0-1}\}$) in the case when \eqref{EQ24}  (resp. \eqref{EQ25}) holds.
Then, proceeding as above, we define the operator $T$ with required properties precisely as in \eqref{EQ23} with the only difference that instead of infinite sequences $\{I_k\}_{k=1}^\infty$, $\{i_k\}_{k=1}^\infty$, $\{\eta'_{i_k}\}_{k=1}^\infty$ and $\{\eta''_{i_k}\}_{k=1}^\infty$
we use the finite ones $\{I_k\}_{k=1}^{k_0}$, $\{i_k\}_{k=1}^{k_0-1}$, $\{\eta_{i_k}'\}_{k=1}^{{k_0-1}}$ and $\{\eta''_{i_k}\}_{k=1}^{{k_0-1}}$.
This completes the proof. 
\end{proof}

\begin{proof}[Proof of Theorem~\ref{Th1}]
If $a=(a_i)_{i=1}^\infty\in \Orb(b;l_0,l_1)$, then from Proposition~\ref{prop1} it follows inequality \eqref{EQ12} with $C=\|a\|_{\Orb}$.

Before proving the converse, let us recall the definition of the {\it dilation operators} in sequence spaces (see, for example, \cite[Sec.~II.8, p.~165]{KPS}). Given $m \in \mathbb N$, by $ {\sigma}_m$ and ${\sigma}_{1/m}$ we set: if $a = (a_n)_{n=1}^\infty$, then
$$
{\sigma}_m a = \left( ( {\sigma}_m a)_n \right)_{n=1}^{\infty} = \big (a_{[\frac{m-1+n}{m}]} \big)_{n=1}^{\infty} 
= \big ( \overbrace {a_1, a_1, \ldots, a_1}^{m}, \overbrace {a_2, a_2, \ldots, a_2}^{m}, \ldots \big)
$$
and
$$
{\sigma}_{1/m} a =  \left( ( {\sigma}_{1/m} a)_n \right)_{n=1}^{\infty} = \Big (\frac{1}{m} \sum_{k=(n-1)m + 1}^{nm} a_k \Big)_{n=1}^{\infty}.$$
It is easy to check that $ {\sigma}_m$ and ${\sigma}_{1/m}$,  $m \in \mathbb N$, are homomorphisms of the groups $l_0$ and $l_1$, $\|{\sigma}_m\|_{l_0 \to l_0}=\|{\sigma}_m\|_{l_1\to l_1}=m$, $\|{\sigma}_{1/m}\|_{l_0 \to l_0}=\|{\sigma}_{1/m}\|_{l_1\to l_1}=1/m$.

Assume first that $C>1$. Then, one can easily check that, by the definition of ${\sigma}_m$,  
$$
C\sum_{i=[k/C]}^\infty b_i^*\le\sum_{i=k}^\infty ({\sigma}_{3([C]+1)}b)_i^*, \;\;k=1,2,\dots
$$
Combining this inequality with \eqref{EQ12}, we get
$$
\sum_{i=k}^\infty a_i^*\le \sum_{i=k}^\infty ({\sigma}_{3([C]+1)}b)_i^*, \;\;k=1,2,\dots
$$
Hence, by Proposition~\ref{prop2}, there exists a homomorphism $Q':\,l_1\to l_1$, $\|Q'\|_{l_1\to l_1}\le 2$, $\|Q'\|_{l_0\to l_0}\le 3$, such that $a=Q'{\sigma}_{3([C]+1)}b$. Since for the homomorphism $Q:=Q'{\sigma}_{3([C]+1)}$, we have $\|Q\|_{l_1\to l_1}\le 6([C]+1)$ and  $\|Q\|_{l_0\to l_0}\le 9([C]+1)$, the proof is completed if $C>1$. 

Let now $C\le 1$. Then, we have 
$$
C\sum_{i=[k/C]}^\infty b_i^*\le\sum_{i=k}^\infty ({\sigma}_{[C^{-1}]^{-1}}b)_i^*, \;\;k=1,2,\dots
$$
Therefore, from \eqref{EQ12} it follows that
$$
\sum_{i=k}^\infty a_i^*\le \sum_{i=k}^\infty ({\sigma}_{[C^{-1}]^{-1}}b)_i^*, \;\;k=1,2,\dots
$$
Reasoning as in the case $C>1$, we get that $a=Qb$ for some homomorphism $Q$ of  the pair $(l_0,l_1)$ such that $\|Q\|_{l_1\to l_1}\le 2[C^{-1}]^{-1}$ and $\|Q\|_{l_0\to l_0}\le 3[C^{-1}]^{-1}$. Thus, the theorem is proved. 
\end{proof}

Corollary~\ref{cor0} is an immediate consequence of Theorem~\ref{Th1}, the definition of the ${\mathbb E}$-orbit, and formula \eqref{EQ13d}.

\begin{proof}[Proof of Corollary~\ref{cor1}]
Let $b=(b_i)_{i=1}^\infty\in l_1$. The embedding
$$
\Orb(b;l_0,l_1)\subset  {\mathcal K}-\Orb(b;l_0,l_1)$$
with constant $1$ follows immediately from the definitions. Therefore, it is left to prove the opposite embedding.

It is well known (see e.g. \cite[Lemma~7.1.3]{BL}) that for every pair of quasi-normed groups $(X_0,X_1)$ and arbitrary $x\in X_0+X_1$ we have
$$
{\mathcal E}^*(t,x;X_0,X_1)=\sup_{s>0}s^{-1}({\mathcal K}(s,x;X_0,X_1)-t),$$
where ${\mathcal E}^*(t,x;X_0,X_1)$ is the greatest convex minorant of  ${\mathcal E}(t,x;X_0,X_1)$, and also that for each $\gamma\in (0,1)$
$$
{\mathcal E}^*(t,x;X_0,X_1)\le {\mathcal E}(t,x;X_0,X_1)\le (1-\gamma)^{-1}{\mathcal E}^*(\gamma t,x;X_0,X_1),\;\;t>0.$$

Assuming now that $x\in {\mathcal K}-\Orb(b;l_0,l_1)$, with $C:=\|x\|_{{\mathcal K}-\Orb}$, and applying the above inequalities for $\gamma=1/2$, we get
\begin{eqnarray*}
{\mathcal E}(2t,x;l_0,l_1)&\le& 2{\mathcal E}^*(t,x;l_0,l_1)=2\sup_{s>0}s^{-1}({\mathcal K}(s,x;l_0,l_1)-t)\\ &\le& 2C\sup_{s>0}s^{-1}({\mathcal K}(s,b;l_0,l_1)-t/C)=2C{\mathcal E}^*(t/C,b;l_0,l_1)\\ &\le& 2C{\mathcal E}(t/C,b;l_0,l_1),\;\;t>0. 
\end{eqnarray*}
Hence,
$$
\sum_{i=k}^\infty x_i^*\le 2C\sum_{i=[k/(2C)]}^\infty b_i^*,\;\;k=1,2,\dots
$$
By Theorem~\ref{Th1}, this implies that $x\in\Orb(b;l_0,l_1)$ and $\|x\|_{\Orb}\le 9([2C]+1)$ if $C>1/2$ and $\|x\|_{\Orb}\le 3[(2C)^{-1}]^{-1}$ if $C\le 1/2$.
\end{proof}

\begin{proof}[Proof of Corollary~\ref{cor2}]
Let first $X$ be a r.i. sequence group such that $l_0\subset X\subset l_1$ and from the inequality 
\begin{equation}\label{EQ25d}
\mathcal K(t,x;l_0,l_1)\le C' \mathcal K(t,b;l_0,l_1),\;\;t>0,
\end{equation}
where $C'$ is a constant, and $b\in X$ it follows that $x\in X$ and $\|x\|_X\le C\|b\|_X$, where $C>0$ does not depend on $x$ and $b$.

If $T$ is a bounded operator in $l_0$ and $l_1$, then for $b\in X$ we have
$$
\mathcal K(t,Tb;l_0,l_1)\le \|T\|_{(l_0,l_1)}\mathcal K(t,b;l_0,l_1),\;\;t>0.$$
Hence, from the hypothesis it follows that $Tb\in X$ and $\|Tb\|_X\le C\|b\|_X$, with some constant $C$ independent of $b$. Thus, $T$ is bounded in $X$, and, as a result, $X$ is an interpolation group with respect to the pair $(l_0,l_1)$.

For the converse, suppose that $X$ is an interpolation group with respect to the pair $(l_0,l_1)$. Let $x\in l_1$ and $b\in X$ satisfy inequality \eqref{EQ25d}. This means that $x\in {\mathcal K}-\Orb(b;l_0,l_1)$. 
Consequently, by Corollary~\ref{cor1}, we have $x\in \Orb(b,l_0,l_1)$. Then, from the definition of $\Orb(b,l_0,l_1)$ it follows that $x=Tb$ for some homomorphism $T$ of the pair $(l_0,l_1)$.  By the interpolation hypothesis, we have $T:\,X\to X$, whence $x\in X$ and $\|x\|_X\le C\|b\|_X$, where $C>0$ does not depend on $x$ and $b$.
\end{proof}

\begin{remark}
\label{E-interpolation}
The following characterization of interpolation groups with respect to the pair $(l_0,l_1)$ is an immediate consequence of Corollary~\ref{cor0}:

Let $X$ be a r.i. sequence group such that $l_0\subset X\subset l_1$.  Then, $X$ is an interpolation group with respect to the pair $(l_0,l_1)$ if and only if from the inequality 
$$
\mathcal E(t,x;l_0,l_1)\le \mathcal E(t,b;l_0,l_1),\;\;t>0,$$
and $b\in X$ it follows that $x\in X$ and $\|x\|_X\le C\|b\|_X$, where $C>0$ depends only on $X$.

Recall that by the $\mathcal E$-functional can be constructed the $\mathcal E$-method of interpolation, which is close to the real interpolation method based on using the $\mathcal K$-functional (see \cite[Sec.~4.2]{BK}). From the above result, in particular, it follows that, for every parameter of the $\mathcal E$-method  $\Psi$, the space ${\mathcal E}_\Psi(l_0,l_1)$ (see  \cite[Definitions~4.2.18 and 4.2.19]{BK}) is an interpolation group with respect to the pair $(l_0,l_1)$ (see Section~\ref{examples} for concrete examples of such groups).
Unfortunately, the $\mathcal K$-divisibility, playing a key role in a description of interpolation spaces with respect to $\mathcal K$-monotone Banach pairs \cite[Sec.~3.2]{BK}, is not longer true in the case of pairs of quasi-normed Abelian groups \cite[Example~3.2.11]{BK}, and this does not allow to describe all interpolation groups with respect to a $\mathcal K$-monotone pair in the above way (as in the Banach case). However, the following weak version of the $\mathcal K$-divisibility for a pair of quasi-normed Abelian groups still holds \cite[Theorem~3.2.12]{BK}:

Let $(X_0,X_1)$ be a pair of quasi-normed Abelian groups, $x\in X_0+X_1$, and let $\varphi_k$, $k=1,2,\dots,N$, be non-negative continuous concave functions on $[0,\infty)$ such that
$$
\mathcal K(t,x;X_0,X_1)\le\sum_{k=1}^N\varphi_k(t),\;\;t>0.$$
Then, there are $x_k\in X_0+X_1$, $k=1,2,\dots,N$, such that
$$
x=\sum_{k=1}^N x_k\;\mbox{and}\;\;K(t,x_k;X_0,X_1)\le\gamma\varphi_k(t),\;\;t>0,\;k=1,2,\dots,N,
$$
where $\gamma$ depends only on $(X_0,X_1)$ and $N$.
\end{remark}

\section{\protect \medskip Groups of Marcinkkiewicz type interpolation with respect to $(l_0,l_1)$}
\label{examples}

Let $(\alpha_k)_{k=1}^\infty$ be an increasing sequence of positive numbers such that for some constants $R_1$ and $R_2$ we have
\begin{equation}\label{EQuat1}
\alpha_{2k}\le R_1\alpha_k,\;\;k=1,2,\dots,
\end{equation}
and
\begin{equation}\label{EQuat2}
\sum_{i=k}^\infty \alpha_{i}^{-1}\le R_2k\alpha_k^{-1},\;\;k=1,2,\dots
\end{equation}
Denote by $M_\alpha$ the set of all sequences $x=(x_k)_{k=1}^\infty$ such that
$$
\|x\|_{\alpha}:=\sup_{k=1,2,\dots}\alpha_kx_k^*<\infty.$$
To prove that this functional defines a r.i. quasi-norm on $M_\alpha$, it suffices to check that for every $x,y\in M_\alpha$ we have
\begin{equation}\label{EQuat3}
\|x+y\|_{\alpha}\le R_1^2(\|x\|_{\alpha}+\|y\|_{\alpha}).
\end{equation}
Indeed, by \cite[Proposition~2.1.7]{BS},
$$
(x+y)_i^*\le x_{[i/2]}^*+y_{[i/2]}^*,\;\;i=2,3,\dots$$
Hence, for all $i=2,3,\dots$ from \eqref{EQuat1} it follows that
$$
\alpha_i(x+y)_i^*\le \alpha_ix_{[i/2]}^*+\alpha_iy_{[i/2]}^*\le \sup_{i=2,3,\dots}\frac{\alpha_i}{\alpha_{[i/2]}}(\|x\|_{\alpha}+\|y\|_{\alpha})\le R_1^2(\|x\|_{\alpha}+\|y\|_{\alpha}).$$
Combining this together the obvious inequality $(x+y)_1^*\le x_1^*+y_1^*$, we see that \eqref{EQuat3} is established. Thus, $M_\alpha$ is a r.i. sequence group.

Let us prove that $M_\alpha$ is an interpolation group with respect to the pair $(l_0,l_1)$. To this end, according to Remark~\ref{E-interpolation}, it suffices to show that 
\begin{equation}\label{EQuat4}
R_1^{-2}R_2^{-1}\|x\|_{\alpha}\le \sup_{k=1,2,\dots}\beta_k^{-1}\sum_{i=k}^\infty x_i^*\le\|x\|_{\alpha},
\end{equation}
where $\beta_k:=\sum_{i=k}^\infty \alpha_i^{-1}$, $k=1,2,\dots$.

First, from the inequality $x_i^*\le \|x\|_{\alpha}\alpha_i^{-1}$, $i=1,2,\dots$, it follows
$$
\sum_{i=k}^\infty x_i^*\le\|x\|_{\alpha}\sum_{i=k}^\infty\alpha_i^{-1}\le \beta_k\|x\|_{\alpha},$$
which implies the right-hand side of inequality \eqref{EQuat4}.

Let now $C:=\sup_{k=1,2,\dots}\beta_k^{-1}\sum_{i=k}^\infty x_i^*<\infty$. Since the sequence $(\alpha_k)_{k=1}^\infty$ increases, by \eqref{EQuat2}, we have for $k=1,2,\dots$
$$
kx_{2k}^*\le \sum_{i=k}^\infty x_i^*\le C\beta_k\le CR_2k\alpha_k^{-1},$$
whence $\alpha_kx_{2k}^*\le CR_2$, $k=1,2,\dots$. Therefore, applying \eqref{EQuat2} once more, for $i=2,3,\dots$ we get
$$
\alpha_ix_{i}^*\le \sup_{k=2,3,\dots}\frac{\alpha_k}{\alpha_{[k/2]}}\alpha_{[i/2]}x_{2[i/2]}^*\le CR_1^2R_2.$$
Combining this with the obvious inequality $\alpha_1x_1^*\le CR_2$, we conclude that $\|x\|_{\alpha}\le CR_1^2R_2$, and thus the left-hand side of \eqref{EQuat4} is proved. 

Let $0<p<1$. Then, $\alpha_k=k^{1/p}$, $k=1,2,\dots$, clearly, satisfy conditions \eqref{EQuat1} and \eqref{EQuat2}. Therefore, the set $M_p$ of all sequences $x=(x_k)_{k=1}^\infty$ such that
$$
\|x\|_{p}:=\sup_{k=1,2,\dots}k^{1/p}x_k^*<\infty$$
is an interpolation r.i. group with respect to the pair $(l_0,l_1)$.

\end{document}